\documentclass{amsart}

\usepackage[utf8]{inputenc} 
\usepackage[T1]{fontenc} 
\usepackage[backend=bibtex,style=numeric,natbib=true,firstinits=true,doi=false,isbn=false,url=false,eprint=false,maxbibnames=99]{biblatex} 
\usepackage[autostyle=true]{csquotes} 
\usepackage[shortlabels]{enumitem}
\usepackage{amsmath}
\usepackage{amsfonts}
\usepackage{amsthm}
\usepackage{tikz-cd}
\usepackage{tikz}
\usepackage{geometry}
\usetikzlibrary{matrix,arrows}

\theoremstyle{plain}
\newtheorem{Th}{Theorem}[section]
\newtheorem{Lemma}[Th]{Lemma}

\theoremstyle{definition}
\newtheorem{Def}[Th]{Definition}

\newtheorem{Rem}[Th]{Remark}
\newtheorem{?}[Th]{Problem}
\newtheorem{Ex}[Th]{Example}

\newcommand{\im}{\operatorname{im}}

\numberwithin{equation}{section}

\usetikzlibrary{decorations.markings}

\makeatletter
\tikzcdset{
	open/.code     = {\tikzcdset{hook, circled};},
	closed/.code   = {\tikzcdset{hook, slashed};},
	open'/.code    = {\tikzcdset{hook', circled};},
	closed'/.code  = {\tikzcdset{hook', slashed};},
	circled/.code  = {\tikzcdset{markwith = {\draw (0,0) circle (.375ex);}};},
	slashed/.code  = {\tikzcdset{markwith = {\draw[-] (-.4ex,-.4ex) -- (.4ex,.4ex);}};},
	markwith/.code ={
		\pgfutil@ifundefined%
		{tikz@library@decorations.markings@loaded}%
		{\pgfutil@packageerror{tikz-cd}{You need to say %
				\string\usetikzlibrary{decorations.markings} to use arrows with markings}{}}{}%
		\pgfkeysalso{/tikz/postaction = {
				/tikz/decorate,
				/tikz/decoration={markings, mark = at position 0.5 with {#1}}}
		}
	},
}
\makeatother

\newcommand{\inv}{^{-1}}

\newcommand{\PP}{\mathbb{P}}

\newcommand{\ZZ}{{\mathbb Z}}
\newcommand{\NN}{{\mathbb N}}
\newcommand{\QQ}{{\mathbb Q}}

\newcommand{\CC}{{\mathbb C}}

\newcommand{\Drm}{\mathrm{D}}

\newcommand{\rH}{\mathrm{H}}
\newcommand{\Hrm}{\mathrm{H}}

\DeclareMathOperator{\gr}{gr}

\DeclareMathOperator{\Proj}{Proj}

\DeclareMathOperator{\Tot}{Tot}

\DeclareMathOperator{\coker}{coker}
\DeclareMathOperator{\Aut}{Aut}
\DeclareMathOperator{\Spec}{Spec}

\newcommand{\lpr}{\,^\prime}
\newcommand{\lprpr}{\,^{\prime\prime}}

\newcommand{\End}{\textrm{End}}

\newcommand{\Hom}{\mathrm{Hom}}

\newcommand{\Oc}{\mathcal{O}}

\newcommand{\Ac}{\mathcal{A}}

\newcommand{\Ic}{\mathcal{I}}

\newcommand{\sbe}{\subseteq}
\newcommand{\spe}{\supseteq}

\newcommand{\OU}{\mathcal{O}_{U}}

\newcommand{\vWv}{\vert W \vert}

\DeclareMathOperator{\id}{id}

\DeclareMathOperator{\Homology}{H}

\author{Philipp Licht}
\address{Philipp Licht \\
	Institut f\"{u}r Mathematik\\
	Johannes Gutenberg-Universit\"{a}t Mainz\\
	Staudingerweg 9, 55099 Mainz\\
	Germany.}
\email{plicht05@uni-mainz.de}



\begin{document}

\title[Infitesimal Torelli for w.c.i. and certain Fano threefolds]{Infinitesimal Torelli for weighted complete intersections and certain Fano threefolds}

\maketitle

\begin{abstract}
	We generalize the classical approach of describing the infinitesimal Torelli map in terms of multiplication in a Jacobi ring to the case of quasi-smooth complete intersections in weighted projective space. As an application, we prove the infinitesimal Torelli theorem for hyperelliptic Fano threefolds of Picard rank 1, index 1, degree 4 and study the action of the automorphism group on cohomology. The results of this paper are used to prove Lang-Vojta's conjecture for the moduli of such Fano threefolds in a follow-up paper.
\end{abstract}


\section{Introduction}

The Torelli problem asks the question if given a family of varieties, whether the period map is injective, i.e., if the variety is uniquely determined by its Hodge structure.  
This question has first been studied for curves; see \cite{Andreotti1958}. 
The infinitesimal Torelli problem is the more general question that asks whether the period map has an injective differential.
The problem can be formulated very concretely for a smooth projective variety $X$ over $\CC$ of dimension $n$. 
Namely, we say that $X$ \emph{satisfies infinitesimal Torelli} if the map 
$$
\Homology^1(X,\Theta_X^1) \to \bigoplus_{p+q = n} \Hom_\CC\left( \Hrm^p(X,\Omega_X^q), \Hrm^{p+1}(X,\Omega_X^{q-1}) \right)
$$
induced by the contraction map is injective. In addition to curves, whether this holds has been studied among others for the following types of varieties:

\begin{itemize}
	\item hypersurfaces in projective space \cite{CaG80, Don83}
	\item hypersurfaces in weighted projective space \cite{Sai86}
	\item complete intersections in projective space \cite{Peters1976, Ter90, Usui1976}
	\item zerosets of sections of vector bundles \cite{Flenner1986}
	\item certain cyclic covers of a Hirzebruch surface \cite{Konno1985}
	\item complete intersections in certain homogeneous Kähler manifolds \cite{Konno1986}
	\item some weighted complete intersections \cite{Usui1977}
	\item certain  Fano  quasi-smooth weighted hypersurfaces \cite{Fatighenti2019}
	\item some elliptic surfaces \cite{Kii1978, Kloosterman2022, SaitoElliptic}
\end{itemize}

The methods used in many of these studies have in common that they describe the cohomology groups relevant for the infinitesimal Torelli map as components of a so-called Jacobi ring
and argue that the map can be interpreted as multiplication by some element in this ring. We generalize this method to the case of quasi-smooth complete intersections in weighted projective space. Following \cite{Dol82}, we introduce the terminology: 

\begin{Def}\label{d:wci}
	Let $ k$ be a  field. For $ W=(W_0,\dots,W_n) \in \NN^{n+1}$ a tuple of positive integers, 
	let $ S_W = k[x_0,\dots,x_n] $ be the graded polynomial algebra with $ \deg(x_i) = W_i $.
	We define \emph{weighted projective space over $k$ with weights} $W$ to be $ \PP(W) = \Proj S_W. $ Given $d = (d_1,\dots,d_c) \in \NN^{c} $, $c \leq n$, a closed subvariety $ X \sbe \PP(W)$ is a \emph{complete intersection of degree $d$} if it has codimension $c$ and is given as the vanishing locus of homogeneous polynomials 
	$ f_1,\dots,f_c \in S_W $ with $\deg(f_i)= d_i$. A weighted complete intersection $X = V(f_1,\dots,f_c) \sbe\PP(W)$ is \emph{quasi-smooth} if its \emph{affine cone} 
	$ A(X) = \Spec S_W/(f_1,\dots,f_c) \setminus\{0\} $ is smooth.
\end{Def}

Given a quasi-smooth weighted complete intersection $X$ over $\CC$ as above, we can define generalized sheaves of differentials $ \tilde{\Omega}^q_{X} $ (see Section allowing a decomposition 
$$
\Hrm^{n-c}(X,\CC) = \bigoplus_{p+q = n-c} \Hrm^p(X,\tilde{\Omega}_X^q)
$$  
that coincides with the usual Hodge decomposition in case $ X $ is smooth; see Theorem \ref{t:hodge}.
Consider the polynomial $ F = y_1 f_1 + \dots + y_c f_c  \in \CC[x_0,\dots,x_{n},y_1,\dots,y_c] $, which is homogeneous with respect to the bigrading given by $ \deg(x_i) = (0,1) $ and $ \deg(y_j) = (1,-d_j) $. 
The \emph{Jacobi ring associated to the complete intersection $X$} is the bigraded ring 
$$
R = \CC[x_0,\dots,x_{n},y_1,\dots,y_c]/(\partial_{x_0}F,\dots,\partial_{x_{n}}F,\partial_{y_1}F,\dots,\partial_{y_c}F).
$$
\subsection*{Main results}

Our first main result can be interpreted as giving an explicit description of the differential of the period map associated to a quasi-smooth weighted complete intersection in terms of its Jacobi ring. 

\begin{Th}\label{t:inf_torelli_map_for_wci}
	Let $ X = V_+(f_1, \dots, f_c) \allowbreak \sbe \PP_\CC(W_0,\dots,W_n) $ be a quasi-smooth weighted complete intersection of degree $(d_1,\dots,d_c)$ with tangent sheaf $\Theta_X^1$ of dimension $ \dim(X) =n-c > 2 $. Let $R$ be the associated Jacobi ring.
	Let $ \nu = \sum W_i - \sum d_j $. For all integers $p \in \ZZ $ with $ 0<p<n-c$ and $p\neq n-c-p$, there are isomorphisms 
	$$
	\Hrm^{n-c-p}(X,\tilde{\Omega}_X^p) \cong \Hom_\CC(R_{p,-\nu},\CC)
	$$
	and 
	$$
	\Hrm^1(X,\Theta_X^1) \cong R_{1,0}.
	$$
	Under these isomorphisms, the contraction map
	$$
	\Hrm^1(X,{\Theta}^1_X) \to \Hom(\Hrm^{n-c-p}(X,\tilde{\Omega}_X^p),\Hrm^{n-c-p+1}(X,\tilde{\Omega}_X^{p-1})) 
	$$
	is the map
	$$
	R_{1,0} \to \Hom(\Hom_\CC(R_{p,-\nu},\CC),\Hom_\CC(R_{p-1,-\nu},\CC)) = \Hom( R_{p-1,-\nu}, R_{p,-\nu})
	$$
	that sends $ \alpha \in R_{1,0}$ to the multiplication-by-$\alpha$ map.
\end{Th}

Our second main result is an application of this theorem to prove the infinitesimal Torelli theorem for (smooth) Fano threefolds of Picard rank $1$, index $1$ and degree $4$.
By Iskovskikh's classification, there are two types of such varieties; see \cite[Table~3.5]{Isk80}.
The varieties of the first type are smooth quartics in $ \PP^4 $. For smooth hypersurfaces in projective space, the infinitesimal Torelli problem is completely understood. In particular, smooth quartic threefolds satisfy infinitesimal Torelli; see \cite{CaG80}. 
The second type of Fano threefolds with Picard rank $1$, index $1$, degree $4$ are called \emph{hyperelliptic}; each such Fano threefold $X$ is a double cover of a smooth quadric $Q \sbe \PP^4 $ ramified along a smooth divisor of degree $8$ in $Q$. Such a double cover comes naturally with an involution $ \iota $ associated to the double cover. 
It turns out that such hyperelliptic Fano threefolds do \textbf{not} satisfy infinitesimal Torelli, i.e., the period map on the moduli of Fano threefolds of Picard rank $1$, index $1$ and degree $4$ does not have an injective differential. However, the following result says that the "restricted" period map on the locus of hyperelliptic Fano threefolds does have an injective differential.

\begin{Th}[Infinitesimal Torelli for hyperelliptic Fano threefolds]\label{t:inf torelli for hyperelliptic Fanos} 
	Let $ X$ be a hyperelliptic (smooth) Fano threefold of Picard rank $1$, index $1$ and degree $4$ over $\CC$. Let $ \iota \in \Aut(X) $ be the involution. Then the $ \iota $-invariant part of the infinitesimal Torelli map 
	$$
	\Homology^1(X,\Theta_X)^\iota \to \bigoplus_{p+q = 3} \Hom_\CC\left( \Hrm^p(X,\Omega_X^q), \Hrm^{p+1}(X,\Omega_X^{q-1}) \right)
	$$
	is injective. 
\end{Th}   

As explained in \cite[Section~3.5]{JLo18}, among the Fano threefolds of Picard number 1 and index 1, infinitesimal Torelli is satisfied if the degree is $2,6$ or $8$, and it is known to fail for degrees $10$ and $14$. Our work deals with one of the remaining cases, namely that of degree $4$. 

Note that the failure of infinitesimal Torelli for Fano threefolds of Picard number 1, index 1 and degree 4 is analogous to the failure of infinitesimal Torelli for curves of genus $ g\geq 2$. Such a curve satisfies infinitesimal Torelli if and only if it is hyperelliptic \cite{Catanese1984}, but the period map restricted to the hyperelliptic locus is an embedding \cite{Landesman2019}.

It is natural to study the action of the automorphism group of a variety on its cohomology group; see for example \cite{Cai2013, Javanpeykar2017, Kuznetsov2018}.
As an application of 
the explicit description of the cohomology groups of a Fano threefold with Picard rank $1$, index $1$, and degree $4$ given by Theorem \ref{t:inf_torelli_map_for_wci},  we get the following result about the action of the automorphism group.

\begin{Th}\label{t:atuomorphism action on cohomology of Fanos}
	Let $ X$ be a (smooth) Fano threefold of Picard rank $1$, index $1$ and degree $4$ over $\CC$. Then the following statements hold. 
	\begin{enumerate}
		\item The automorphism group $ \Aut(X) $ acts faithfully on $ \Homology^1(X,\Theta_X) $. 
		\item If $X$ is a smooth quartic, then $ \Aut(X) $ acts faithfully on $ \Hrm^3(X,\CC)$.
		\item If $X$ is hyperelliptic, then the kernel $ \ker\left(\Aut(X) \to \Aut(\Hrm^3(X,\CC))\right)$ is isomorphic to $ \ZZ/2\ZZ $ and generated by the involution $\iota$.
	\end{enumerate}
\end{Th}

Note that part (2) was already known; see for example \cite[Proposition~2.12]{JaLoCI2017}.

\subsection*{Ingredients of proof}

For smooth complete intersections $X =V(f_1,\dots,f_c)$ in usual projective space,
similar results to Theorem \ref{t:inf_torelli_map_for_wci} have been achieved by relating the IVHS of $ X $ to the IVHS of the hypersurface $V(F) \sbe \PP(E) $, with $ E= \bigoplus \Oc_{\PP^n}(d_i) $; see \cite{Ter90}. 
To avoid problems of this geometric approach arising from the singular nature of the surrounding weighted projective space in our case, we will use another purely algebraic approach inspired by the calculations of Flenner; see \cite[Section~8]{Fle81}.
We will construct resolutions of the sheaves $ \tilde{\Omega}_X^p $ that will give us spectral sequences converging towards the cohomology groups of interest. 
The difficult part will be to make sure that the identification of the cohomology parts with the homogeneous components of the Jacobi ring is done in such a way that the contraction map can be identified with the ring-multiplication. To do this, we will extend the contraction pairing to a pairing of the resolutions and then to a pairing of the spectral sequences.   

\subsection*{Arithmetic motivation}
It is well-known that a variety admitting a quasi-finite period map is hyperbolic \cite[Section~8-9]{Griffiths1969}, and therefore (by Lang-Vojta's conjecture) should have only finitely many integral points, i.e., be "arithmetically hyperbolic"; see \cite[§~0.3]{Abramovich1997} or \cite{JBook, Lang1986}. For evidence on Lang-Vojta's arithmetic conjectures, see \cite{Autissier1, Autissier2, Corvaja2006, Faltings1994, Levin2009, Javanpeykar2021,UllmoShimura}.

We were first led to investigate the infinitesimal Torelli problem for these Fano threefolds when studying the arithmetic hyperbolicity of the moduli stack $ \mathcal{F} $ of Fano threefolds of Picard rank 1, index 1, degree 4; see \cite[Section~2]{JLo18} for a definition of this stack. The property of a stack being arithmetically hyperbolic, i.e., having "only finitely many integral points" is formalized in \cite{JLChevalleyWeil2020}.

In \cite{Licht2022}, we prove the arithmetic hyperbolicity of this stack by first proving that the period map 
$$
p\colon \mathcal{F}^{an}_\CC  \to \Ac_{30}^{an} 
$$ 
is quasi-finite and then using Faltings's theorem \cite{Faltings1983} which says that the stack of principally polarized abelian varieties $ \Ac_{30}  $ is arithmetically hyperbolic.
For the cases of Fano threefolds of Picard rank 1, index 1 and degree $2,6$ or $8$, the quasi-finiteness of the period map is deduced from it being unramified, i.e. its differential, the infinitesimal Torelli map, being injective; see \cite{JLo18}. However, by our result in the degree $4$ case, the infinitesimal Torelli map is not injective. We overcome this difficulty in \cite{Licht2022} by showing that the moduli stack $\mathcal{F}$ has a natural two-step "stratification" and that on each stratum, the "restricted" period map is unramified. 
This then suffices to deduce the desired quasi-finiteness of the above period map.

\subsection*{Acknowledgements}

I would like to thank Ariyan Javanpeykar. He introduced me to Lang-Vojta's conjecture. The work presented here was done under his supervision during my phd project. I am very grateful for many inspiring discussions and his help in writing this article. I gratefully acknowledge the support of
SFB/Transregio 45.


\section{Multigraded differential modules}
\label{s:graded_modules}
In this section, we introduce multi-graded differential modules, which is a notion used for example in \cite{Boocher2010}. In particular, this notion describes single and double complexes and pages of spectral sequences.  

Let $R$ be a (commutative) ring or, more generally,  the structure sheaf $ \Oc_T $ of a scheme $T$. A differential $d$ on an $n$-graded $R$-module $ E = \bigoplus_{p \in \ZZ^n} E^{p} $ for us is always considered to be an $R$-linear self map that is homogeneous of a certain degree with $d \circ d = 0$. Whenever we consider a module together with multiple differentials defined on it, we require the differentials to commute pairwise. 

Let $(E_1,d_1) $ and $ (E_2, d_2 )$ be differential $n$-graded $R$-modules with homogeneous differentials of the same degree $ a\in \ZZ^n $. Then the tensor product $ E_1 \otimes E_2 $ comes with an induced $(2n)$-grading 
$$
E_1 \otimes  E_2 = \bigoplus_{(p,q)\in \ZZ^n \times \ZZ^n} E^p_1 \otimes E^{q}_2
$$  
and the two homogeneous differentials $ d_1 \otimes \id $ and $ \id \otimes d_2 $, giving us a bidifferential $2n$-graded $R$-module.  

\begin{Def}
	Let $n\in \ZZ_{>0} $ be a positive integer and let $(E,d_1,d_2)$ be a bidifferential $2n$-graded $R$-module. 
	Write the degree of $ d_i $ as $ (a_i,b_i)$ where $ a_i,b_i \in \ZZ^n $. Suppose $a_1+b_1 = a_2+b_2$, then 
	we define the \emph{associated total differential $n$-graded $R$-module of $ (E,d_1,d_2) $}  to be the $n$-graded module
	$$
	\Tot(E) = \bigoplus_{p \in \ZZ^n} \Tot^p(E) 
	$$
	where 
	$$
	\Tot^p(E) = \bigoplus_{\substack{s,t\in \ZZ^n \\ s+t = p}} E^{s,t}	
	$$
	with homogeneous differential $ d \in \End(\Tot(E) )$ of degree $a_1+b_1 = a_2+b_2 $
	defined by 
	$
	d\vert_{E^{s,t}} = d_1 + (-1)^{s_1} d_2.
	$
\end{Def}

\begin{Ex}
	Let $(K^{\bullet,\bullet},d_1,d_2)$  be a double complex. Then $ K = \bigoplus_{p,q\in \ZZ} K ^{p,q} $
	is a bigraded module and $d_1$, $d_2$ define differentials of degree $(1,0)$, $(0,1)$ on $K$, thus giving $K$ the structure of a bidifferential bigraded module. In fact, giving the data of a double complex is equivalent to defining a bigraded module with 
	differentials of degree $(1,0)$  and $(0,1)$.
	Similarly, a complex $(L^\bullet,d) $ can be identified with the differential graded module $ (L = \bigoplus_{p\in \ZZ} L^p, d)$.
	Under these identifications, the total single complex associated to $ K^{\bullet,\bullet}$ and the total differential graded module associated to $K^{\bullet,\bullet}$ are the same. 
\end{Ex}

\begin{Ex}
	For us, the total differential bigraded module associated to a tensor product of bigraded differential modules with differentials of the same degree $a\in \ZZ^2$ is of particular interest.
	So let $(E_1,d_1) $ and $ (E_2, d_2 )$ be differential bigraded modules. 
	Then the differentials $ d_1 \otimes \id $ and $ \id \otimes d_2 $ on the quadgraded module $ E_1 \otimes E_2 $ have degrees $ (a_1,a_2,0,0)$ and $(0,0,a_1,a_2)$.
	For $p,q \in \ZZ $, we have
	$$
	\Tot^{p,q}(E_1 \otimes E_2) = \bigoplus_{\substack{s+t = p \\ u+v = q}} E^{s,u}_1 \otimes  E^{t,v}_2.
	$$
	On $  E^{s,u} \otimes  E^{t,v}$  the differential is given as 
	$$
	d_{Tot}\vert_{E^{s,u} \otimes  E^{t,v}} = d_1\vert_{E^{s,u}} \otimes \id_{E^{t,v}} + (-1)^s \id\vert_{E^{s,u}} \otimes  d_2\vert_{E^{t,v}}.
	$$
\end{Ex}


\section{Pairings of filtered complexes}\label{s:pairings of filtered complexes}

In this section, we explain how a pairing of filtered complexes induces a pairing of the associated homology complexes that respects the induced filtration.		
Let  $R$ be a  ring or, more generally,  the structure sheaf $ R= \Oc_T $ of a scheme $T$. All modules are considered to be $R$-modules and all single (resp. double) complexes are considered to be single (resp. double) complexes of $R$-modules.

Let $ (K,d) $, $ ( K^\bullet_1, d_1) $ and $ ( K^\bullet_2,  d_2) $ be complexes. The total complex of the tensor product of $K^\bullet_1$ and $K^\bullet_2$, as introduced in Section \ref{s:graded_modules}, is given by
$$
\Tot^n( K^\bullet_1 \otimes  K^\bullet_2) = \bigoplus_{p+q=n}  K^{p}_1 \otimes K^{q}_2
$$
with the differential given by
$$
d_{Tot}\vert_{K_1^p\otimes K_2^q} = d_1 \otimes \id\vert_{K_2^q} + (-1)^p \id\vert_{K_1^p} \otimes  d_2.
$$

\begin{Def} 
	A \emph{pairing of complexes from $ ( K^\bullet_1, d_1) $ and $ ( K^\bullet_2,  d_2) $ to $ (K,d) $} is a morphism of complexes
	$$
	\phi:(\Tot^\bullet( K^\bullet_1 \otimes  K^\bullet_2) ),d_{Tot})\to (K^\bullet,d).
	$$
	For $p,q \in \ZZ $, we let  $ \phi^{p,q}$ denote the map 
	$
	\phi^{p,q}\colon K^{p}_1 \otimes K^{q}_2 \to K^{p+q}
	$ 
	induced by $\phi$.
	
\end{Def}

From now on let $R$ be a ring. Directly from the definitions follows:

\begin{Lemma}\label{l:pairing of complexes induces pairing of homology}
	Let $ (K^\bullet,d)$, $ ( K^\bullet_1, d_1) $ and $ ( K^\bullet_2,  d_2) $ be complexes and let 
	$$
	\phi\colon 	\Tot^n( K^\bullet_1 \otimes  K^\bullet_2)  \to K^\bullet
	$$
	be a pairing of complexes. Then $\phi$ induces a pairing of the associated homology complexes
	$$
	\overline{\phi}\colon \Tot^\bullet( \Homology^\bullet(K^\bullet_1,d_1) \otimes  \Homology^\bullet(K_2, d_2) ) \to \Homology^\bullet(K^\bullet,d).
	$$
\end{Lemma}


A \emph{filtered complex} is a triple $(K^\bullet,d,F)$, where $K^\bullet$ is a complex with differential $ d$,  and $ F $ is a decreasing filtration on $K^\bullet$ compatible with the differential, 
i.e., for each $n\in \ZZ$, we have a decreasing filtration 
$$
K^n \spe \dots \spe F^{p} K^n \spe F^{p+1} K^n \spe \dots  
$$
such that $ d(F^{p} K^n) \sbe F^{p} K^{n+1} $ for all $n,p \in \ZZ$.

Given a filtered complex $(K^\bullet, d, F)$, there is an \emph{induced filtration} on the homology complex $ \Hrm^\bullet(K^\bullet,d)$ given by 
$$
F^p \Hrm^n (K^\bullet,d)) := \im(\Hrm^n(F^p K^\bullet) \to \Hrm^n(K^\bullet)) = \frac{\ker(d)\cap F^p K^n + \im(d) \cap K^n}{\im(d) \cap K^n}.
$$ 
For the associated graded, we have
\begin{equation}\label{eq:grpHnK}
	\mathrm{gr}^p \Hrm^n(K^\bullet) := \frac{F^p \Hrm^n (K^\bullet,d))}{F^{p+1} \Hrm^n (K^\bullet,d)} = \frac{\ker(d)\cap F^p K^n}{\ker(d)\cap F^{p+1} K^n + \im(d) \cap F^p K^n}; 
\end{equation}
see \cite[Tag~0BDT]{stkpr}.

\begin{Def}
	Let $(K^\bullet,d,F)$,$( K^\bullet_1, d_1, F_1)$ and $ ( K^\bullet_2, d_2, F_2)$ be filtered complexes. A pairing of complexes $$
	\phi\colon \Tot^\bullet( K^\bullet_1 \otimes  K^\bullet_2) \to K^\bullet
	$$
	is a \emph{pairing of filtered complexes} if it is \emph{compatible with the filtrations}, that is
	$$
	\phi^{p,q}(\alpha \otimes \beta) \in F^{i+j} K^{p+q}
	$$
	for all $\alpha \in  F^i_1  K^p_1 $ and $ \beta \in  F^j_2 K^q_2 $.
\end{Def}

From the definition, it is evident that for each $ p,q \in \ZZ $, such a pairing of filtered complexes induces a pairing of complexes 
$$
\phi\colon \Tot^\bullet( F^p_1 K^\bullet_1 \otimes F^{q}_2  K^\bullet_2) \to F^{p+q} K^\bullet.
$$
Hence the induced pairing of the homology complexes 
$$
\overline{\phi}\colon \Tot^\bullet (\Hrm^\bullet ( K^\bullet_1, d_1 ) \otimes \Hrm^\bullet {( K^\bullet_2,d_2)}) \to \Hrm^\bullet (K^\bullet,d)
$$
from Lemma \ref{l:pairing of complexes induces pairing of homology} is compatible with the induced filtrations on the homology complexes. Therefore, for each $p,q,i,j \in \ZZ $, we get induced maps $ \overline{\phi}^{p,q,i,j} $ and $ \gr^{p,q,i,j}(\phi)$
making the diagram 
\begin{equation}\label{e:diagram}
	\begin{tikzcd}[column sep=9ex]
		\Hrm^{p} ( K^\bullet_1) \otimes \Hrm^{q} (K^\bullet_2)   \ar[d,"\overline{\phi}^{p,q}"] & 
		F^i\Hrm^{p} ( K^\bullet_1) \otimes F^j\Hrm^{q} (K^\bullet_2)\ar[d,"\overline{\phi}^{p,q,i,j}"] \ar[l,"\alpha^{p,i}\otimes \alpha^{q,j}"] \ar[r,"\beta^{p,i}\otimes \beta^{q,j}"]& 
		\gr^i\Hrm^{p} ( K^\bullet_1) \otimes \gr^j\Hrm^{q} (K^\bullet_2) \ar[d,"\gr^{p,q,i,j}(\phi)"]\\
		\Hrm^{p+q} ( K^\bullet)  & 
		F^{i+j}\Hrm^{p+q} ( K^\bullet) \ar[l,"\alpha^{p+q,i+j}"] \ar[r,"\beta^{p+q,i+j}"]  & 
		\gr^{i+j}\Hrm^{p+q} ( K^\bullet) 
	\end{tikzcd}
\end{equation}
commute, where the maps $ \alpha^{a,b}$ denote the natural injections and the maps $\beta^{a,b} $ denote the natural surjections.

\label{s:pairings_of_filtered_complexes}


\section{Spectral pairing} \label{s:spectral pairing}

In this section, we follow \cite[Tag~012K]{stkpr} and explain how to construct the spectral sequence associated to a filtered complex. Building on this, we show that a pairing of filtered complexes induces a pairing of the associated spectral sequences. 

Let $R$ be a ring. All complexes are complexes of $R$-modules. 
A \emph{spectral sequence} is given by the data 
$$
E = (E_r,d_r )_{r \in \ZZ_{\geq 0}}
$$
where $ E_r $ is an $R$-module and $ d_r \in \End(E_r) $ is a differential such that $$ E_{r+1} = \Homology (E_r,d_r) := \ker(d_r)/ \im(d_r). $$ We call $E_r$ the  \emph{$r$-th page of $E$}. A \emph{bigrading on the spectral sequence $E$} is given by a direct sum decomposition for each page $ E_r = \bigoplus_{p,q\in \ZZ} E_r^{p,q} $ such that the differential $d_r$ decomposes into a direct sum of maps
$$ d_r^{p,q}\colon E_r^{p,q} \to E_r^{p+r,q-r+1} $$ and we have 
$$
E^{p,q}_{r+1} = \ker(d_r^{p,q}) / \im d_r^{p-r,q+r-1}.
$$

 We can associate a bigraded spectral sequence to a filtered complex $(K^\bullet,d,F)$ in the following way. We define
$$
Z_r^{p,q} = \frac{F^p K^{p+q}\cap d\inv(F^{p+r}K^{p+q+1})+F^{p+1}K^{p+q}}{F^{p+1}K^{p+q}}
$$
and
$$
B_r^{p,q} = \frac{F^pK^{p+q}\cap d(F^{p-r+1}K^{p+q-1})+F^{p+1}K^{p+q}}{F^{p+1}K^{p+q}}
$$
and $E_r^{p,q} = Z_r^{p,q}/B_r^{p,q}$. Now set $ B_r = \bigoplus_{p,q} B_r^{p,q} $, $ Z_r = \bigoplus_{p,q} Z_r^{p,q} $ and $ E_r = \bigoplus_{p,q} E_r^{p,q} $. 
Define the map $ d_r\colon E_r \to E_r $ as the direct sum of the maps $$ d_r^{p,q}\colon E_r^{p,q} \to E_r^{p+r,q-r+1}\colon \quad z+ F^{p+1}K^{p+q} \mapsto d(z) + F^{p+r+1} K^{p+q+1} $$ where $z \in F^p K^{p+q}\cap d\inv(F^{p+r}K^{p+q+1})$. This defines the \emph{bigraded spectral sequence $ (E_r,d_r) $ associated to the filtered complex
	$(K^\bullet,d,F)$}.

\begin{Def}
	Let $ ( E_r,d_r), (\lpr E_r, \lpr d_r)$ and $ (\lprpr E_r, \lprpr d_r) $ be bigraded spectral sequences. Let $\phi = (\phi_r)_{r \in \ZZ_{\geq 0}} $ be a collection of morphisms of bigraded differential modules 
	$$
	\phi_r\colon \Tot^{\bullet,\bullet}(\lpr E_r^{\bullet,\bullet} \otimes \lprpr E_r^{\bullet,\bullet}) \to E_r^{\bullet,\bullet} 
	$$
	that are homogeneous of degree $0$. 
	We denote by  
	$$  \phi^{s,t,u,v}_r\colon \lpr E_r^{s,u} \otimes 
	\lprpr E_r^{t,v} \to E_r^{s+t,u+v}$$ the map induced by $\phi$.
	The collection $ \phi $ is called a \emph{pairing of bigraded spectral sequences} if 
	$ \phi_{r+1}^{s,t,u,v} $ is induced by $ \phi_r^{s,t,u,v} $ for all $r,s,t,u,v \in \ZZ$, $r \geq 0 $. 
	
\end{Def}

Form the definitions, we see: 

\begin{Lemma}\label{l:pairing_of_filtered_complexes_induces_pairing_of_spectral_sequences}
	Let $((K^\bullet, d, F), (E_r,d_r)) $, $((\, ^\prime K^\bullet, \,^\prime d , \,^\prime F),  ( \lpr E_r,\lpr d_r)) $ and $((\,^{\prime\prime} K^\bullet, \,^{\prime\prime} d ,  \,^{\prime\prime} F),  (\lprpr E_r, \lprpr d_r))$
	be pairs of filtered complexes with their associated bigraded spectral sequences. 
	Any pairing of filtered complexes 
	$$ \phi:\Tot^\bullet( ( \, ^\prime K^\bullet, \,^\prime d , \,^\prime F) \otimes (\,^{\prime\prime} K^\bullet, \,^{\prime\prime} d , \,^{\prime\prime} F) )\to (K^\bullet, d, F) $$ 
	induces a pairing of the associated spectral sequences $ \tilde{\phi} = (\tilde{\phi}_r)_{r\in\ZZ_{\geq 0}} $,
	$$ \tilde{\phi}_r \colon \Tot^{\bullet,\bullet}((\lpr E_r^{\bullet,\bullet}, \lpr d_r)  \otimes (\lprpr E_r^{\bullet,\bullet}, \lprpr d_r)) \to  ( E_r^{\bullet,\bullet},d_r). $$ 
\end{Lemma}

\begin{proof}
	A computation shows that $\phi$ induces a map $ \lpr Z_r^{s,u} \otimes \lprpr Z_r^{t,v} \to Z_r^{s+t,u+v} $  which maps both $ \lpr B_r^{s,u} \otimes \lprpr Z_r^{t,v} $ and $ \lpr Z_r^{s,u} \otimes \lprpr B_r^{t,v} $ to $ B_r^{s+t,u+v} $.
	That $ \tilde{\phi}_{r+1}^{s,t,u,v} $ is induced by $ \tilde{\phi}_r^{s,t,u,v} $ follows from the fact that both maps are induced by $ \phi $. For details see \cite{LichtPdh2022}.
\end{proof}


For a filtered complex $(K^\bullet,d,F)$, we define 
$$
Z_\infty^{p,q} = \bigcap_r Z_r^{p,q} = \bigcap_r  \frac{F^p K^{p+q}\cap d\inv(F^{p+r}K^{p+q+1})+F^{p+1}K^{p+q}}{F^{p+1}K^{p+q}}
$$
and
$$
B_\infty^{p,q} = \bigcup_r B_r^{p,q} = \bigcup_r \frac{F^pK^{p+q}\cap d(F^{p-r+1}K^{p+q-1})+F^{p+1}K^{p+q}}{F^{p+1}K^{p+q}}
$$
and $ E_\infty^{p,q} = Z_\infty^{p,q} / B_\infty^{p,q} $.
If we now suppose that the filtration is finite, i.e., for all $ n \in \ZZ $, there are $l, m\in \ZZ$ such that $ F^l K^n = K^n $ and $ F^m K^n = 0 $, then the chains 
$$
Z_0^{p,q} \spe \dots Z_r^{p,q} \spe Z_{r+1}^{p,q} \spe \dots 
$$
and
$$
B_0^{p,q} \sbe \dots B_r^{p,q} \sbe B_{r+1}^{p,q} \sbe \dots 
$$
become stationary and assume $Z_\infty^{p,q}$ and $ B_\infty^{p,q}$ after finitely many steps. We have
$$
Z_\infty^{p,q} = \frac{F^p K^{p+q}\cap \ker(d) +F^{p+1}K^{p+q}}{F^{p+1}K^{p+q}}
$$
and
$$
B_\infty^{p,q} =  \frac{F^p K^{p+q}\cap \im(d) + F^{p+1}K^{p+q}}{F^{p+1}K^{p+q}}.
$$

If we now put $n=p+q$ and compare with Equation (\ref{eq:grpHnK}), we get an identity 
\begin{equation}\label{e:E infinity is graded}
	\mathrm{gr}^p \Hrm^n(K^\bullet) = \frac{\ker(d)\cap F^p K^n + F^{p+1} K^n}{\im(d)\cap F^p K^n + F^{p+1} K^n} = E_\infty^{p,q}.
\end{equation}

\begin{Th}\label{t:spectral pairing}
	Let $((K^\bullet, d, F), (E_r,d_r)) $, $((\, ^\prime K^\bullet, \,^\prime d , \,^\prime F),  ( \lpr E_r,\lpr d_r)) $ and $((\,^{\prime\prime} K^\bullet, \,^{\prime\prime} d ,  \,^{\prime\prime} F), (\lprpr E_r, \lprpr d_r))$
	be pairs of filtered complexes with their associated bigraded spectral sequences such that all the filtrations are finite, and let $$
	\phi\colon \Tot^\bullet(\lpr K^\bullet \otimes \lprpr K^\bullet) \to K^\bullet
	$$
	be a pairing of filtered complexes. The induced pairing of the associated spectral sequences induces a pairing of bigraded modules 
	$$
	\tilde{\phi}_\infty\colon \Tot^\bullet(\lpr E_\infty^{\bullet,\bullet} \otimes \lprpr E_\infty^{\bullet,\bullet}) \to E_\infty^{\bullet,\bullet}
	$$
	such that for all $i,j,p,q \in \ZZ $, the diagram
	$$
	\begin{tikzcd}[column sep=13ex]
		\lpr E_\infty^{i,p} \otimes \lprpr E_\infty^{j,q} \ar[r,"="] \ar[d,"{\tilde{\phi}_\infty^{i,j,p,q}}"]  &  \gr^i\Homology^{p+i}(\lpr K^\bullet) \otimes \gr^j \Homology^{q+j} (\lprpr K^\bullet) \ar[d,"{\gr^{p+i,q+j,i,j} (\phi)}"] 
		\\
		E_\infty^{i+j, p+q}   \ar[r,"="] &
		\gr^{i+j}\Homology^{p+q+i+j}(K^\bullet)
	\end{tikzcd}
	$$
	commutes.
\end{Th}

\begin{proof}
	By Lemma  \ref{l:pairing_of_filtered_complexes_induces_pairing_of_spectral_sequences}, the pairing of filtered complexes $\phi$ induces a pairing of spectral sequences $$ \tilde{\phi}_r \colon \Tot^{\bullet,\bullet}((\lpr E_r^{\bullet,\bullet}, \lpr d_r)  \otimes (\lprpr E_r^{\bullet,\bullet}, \lprpr d_r)) \to  ( E_r^{\bullet,\bullet},d_r). $$ 
	For each ${i,j,p,q}\in \ZZ $ the modules $ E_r^{i,p} $, $ 'E_r^{i,p} $ and $ ''E_r^{i,p} $ assume  $ E_\infty^{i,p} $, $ 'E_\infty^{i,p} $ and $ ''E_\infty^{i,p} $ after finitely many pages. Hence the maps $ \tilde{\phi}^{i,j,p,q}_r$ converge to a map $$
	\tilde{\phi}^{i,j,p,q}_\infty\colon  \lpr E_\infty^{i,p} \otimes \lprpr E_\infty^{j,q} \to E_\infty^{i+j, p+q}.
	$$  
	It coincides with $\gr^{p+i,q+j,i,j} (\phi)$ as both maps are induced by $\phi$.
\end{proof}


\section{The contraction pairing on the affine cone}\label{s:pairing of resolutions}

Let $k$ be a field of characteristic zero and let $X = V_+(f_1,\dots,f_c) \sbe \PP_k(W_0,\dots,W_n) $ be a quasi-smooth weighted complete intersection of degree $(d_1,\dots,d_c)$ with coordinate ring $$ A = S_W/(f_1,\dots,f_c) $$ and affine cone $ U = Y \setminus \{0 \}$, where $ Y = \Spec A $. Let $ \Ic \sbe \Oc_{\mathbb{A}^{n+1}\setminus\{0\}}$ be the ideal sheaf of $ U $ in $\mathbb{A}^{n+1} \setminus \{0\} $. 
Let $\Omega_U^1$ be the sheaf of $k$-differentials on $U$ and let $ \Theta_U^1 $ be its dual, namely the tangent sheaf. 
Let $ p $ be an integer satisfying $1 \leq p \leq n-c $.
Building on Flenner's calculations \cite[Section~8]{Fle81}, in this section we will construct free resolutions of the sheaves $ \Omega_U^p $ and extend the contraction pairing $$
\Omega_U^p \otimes \Theta_U^1 \xrightarrow{\gamma} \Omega_U^{p-1}
$$ 
to these resolutions and their associated total \v Cech cohomology complexes. 
\subsection*{The resolutions}
The conormal sequence associated to the closed immersion of the smooth complete intersection $U$ into $ \mathbb{A}^{n+1}\setminus \{ 0 \}$, namely 
$$
0 \to \Ic/\Ic^2 \to \Omega_{\mathbb{A}^{n+1}\setminus\{0\}}^1 \otimes \Oc_U \to \Omega_U^1 \to 0,
$$
is exact and $\Omega_U^1$ is locally free; see \cite[Theorem~II.8.17]{Hartshorne2008}. This uses the smoothness of $U$.
The $ \Oc_U$-module $\Omega_{\mathbb{A}^{n+1}\setminus\{0\}}^1 \otimes \Oc_U $
is free of rank $n+1$ and spanned by the elements $ dx_0,\dots, dx_n $.
The conormal sheaf $ \Ic/\Ic^2$ of the complete intersection $U$ is free of rank $c$ and is generated by the elements $f_1,\dots,f_c$. Hence the conormal sequence 
is described by the exact sequence
\begin{equation} \label{e:conormal complex}
	0 \to F = \bigoplus\limits_{i\in\{1,\dots,c\}} \OU   \cdot y_i  \xrightarrow{\phi}  G =\bigoplus\limits_{j \in\{ 0, \dots n\}} \OU \cdot dx_j \xrightarrow{\pi}  \Omega_U^1 \to 0
\end{equation}
of $\OU$-modules, 
where the $y_i$ are basis elements, the morphism $\phi$ is the $\OU$-linear map with $$ \phi(y_i) = d(f_i) = \sum_{j=0}^n \partial_{j}(f_i) \cdot dx_j $$ and $\pi$ is the natural surjection. Note, if we set $ \deg(y_i) = d_i$ and $ \deg(dx_i) = w_i $, then the induces morphisms $ \phi_U: \Gamma(U,F) \to \Gamma(U,G) $ and $ \pi_U: \Gamma(U,G) \to \Gamma(U,\Omega_U^1) $ are homogeneous of degree $0$.

For any quasi-coherent $\Oc_U$-module $N$ and $r\in \ZZ_{\geq 0}$, let $ S^r(N) $ denote the $r$-th symmetric power of $N$. 
As the $\Oc_U$-module $F$ is free with a basis $y_1,\dots,y_c $, the symmetric power $ S^r(F) $ is free with a basis formed by the elements
$$
y^{ \lambda} := y_1^{ \lambda_1}\cdot \ldots \cdot y_c^{ \lambda_c}
$$
where $ \lambda \in \ZZ_{\geq 0}^c$ with $\sum  \lambda_i = r$. For the notation $y^{ \lambda}$, we will allow $  \lambda \in \ZZ^c $. Namely,  if $ \lambda_i < 0$ for some $ i $, then we set $ y^{ \lambda} = 0$.  
Similarly to \cite[example~(ii)]{Leb77}, we define the complex
$ (K_p^\bullet,d_{K_p}^\bullet) $ of $\OU$-modules with components $$ K_p^q = S^{-q}(F) \otimes \bigwedge^{p+q} (G) $$ for $-p\leq q \leq 0$ and $K_p^q = 0 $ otherwise and differential 
$$
K_p^q = S^{-q}(F) \otimes \bigwedge^{p+q} (G) \to  K_p^{q+1} = S^{-q-1}(F) \otimes \bigwedge^{p+q+1} (G)
$$
given as the $\OU$-linear map that sends $ y^{\lambda} \otimes \omega $, where $ \lambda \in \ZZ_{\geq 0}^c $ with $ \sum_{i=c} \lambda_i = -q  $ and $ \omega = dx_{i_1}\wedge \dots \wedge dx_{i_{p+q}} $, to
$$
\sum_{i=1}^c y^{\lambda-e_i} \otimes d(f_i) \wedge \omega,
$$
where $ e_i \in \ZZ^c $ denotes the $i$-th standard basis vector. 

By composing it with the natural surjection $ K_p^0 = \bigwedge^p (G) \to \Omega_U^p $,
we get a complex
$$ 0 \to K_p^{-p} \to \dots \to K_p^{0} \to \Omega_U^p  \to 0. $$ 
Note for $ p = 1 $ this is Sequence (\ref{e:conormal complex}).
By dualizing the exact sequence (\ref{e:conormal complex}) of locally free sheaves, we get an exact sequence
$$
0 \to \Theta_U^1 \xrightarrow{\pi^*} G^* = \bigoplus_{i=0}^n \OU \cdot \delta_{i} \xrightarrow{\phi^*}  F^* = \bigoplus_{j=1}^c \OU \cdot y_j^* \to 0,
$$
where the elements $ \delta_0,\dots, \delta_n $ are the dual basis for $ dx_0,\dots,dx_n$ and the elements $ y_1^*,\dots,y_c^*$ are the dual basis for $y_1,\dots,y_c$.
The differential $ \phi^* $  maps $ \delta_i $ to $ \sum_{j=1}^c \partial_{x_i}(f_j) \cdot y_j^*$.
Again, we set $ \deg(\delta_i)= -w_i$ and $ \deg(y_i^*) = -d_i$, so that the sequence becomes homogeneous of degree $0$ on global sections.
We define the complex $ (K_{-1}^\bullet,d_{K_{-1}}^\bullet) $ with components 
$ K_{-1}^0 = G^*$, $ K_{-1}^1 = F^* $ and $ K_{-1}^q = 0 $ if $q\not\in \{0,1\}$ and differential $ \phi^* $. These complexes give the desired resolutions.

\begin{Th}\label{t:Kp is exact}
	In the situation above, for every $p \in \{1,\dots, n-c \} $,
	the complex of $\Oc_U$-modules
	$$ 0 \to K_{p}^{-p} \to \dots \to K_{p}^{0} \to \Omega_U^p  \to 0 $$
	is exact. Furthermore the complex of $\Oc_U$-modules
	$$ 
	0\to \Theta_U^1 \to K_{-1}^0 \to K_{-1}^1 \to 0
	$$
	is exact.	
\end{Th} 

\begin{proof}
	We have already proven the second statement and the first statement for $p=1$. Let $ p>1$ and let $ V = \Spec B \sbe U $ be any affine open such that the sheaf $ \Omega^1_U $ restricted to $ V $ is free. Let 
	$ M = \Gamma(V,\Omega_U^1) $ be a free $B$-module. Hence it is $m$-torsion-free (see \cite[Introduction]{Leb77} for definition) for any positive integer 
	$ m\in \ZZ_{>0} $. Hence by applying \cite[Theorem~3.1]{Leb77} 
	(note, since $\mathrm{char}(k) = 0$, the ring $B$ is a $\QQ$-algebra and hence the divided powers used in that reference are isomorphic to symmetric powers) to $ M $ with the free resolution 
	$$
	0 \to \Gamma(V,K_{1}^{-1}) \to  \Gamma(V,K_{1}^{0}) \to M \to 0, 
	$$
	we see that the complex 
	$$
	0 \to \Gamma(V,K_{p}^{-p}) \to \dots \to \Gamma(V,K_{p}^{0}) \to \Gamma(V,\Omega_U^p)  \to 0 
	$$
	is exact. Since  $ \Omega^1_U $ is locally free, we can cover $ U $ with affine opens $V$ such that the restriction is free. So we are done.  
\end{proof}

\subsection*{The pairing of resolutions}
There are $\OU$-bilinear contraction maps 
\begin{align*}
	\tilde{\gamma}_G\colon \quad \bigwedge^q (G) \times G^* &\to  \bigwedge^{q-1} (G) \\ 
	(dx_{i_1} \wedge \dots \wedge dx_{i_q}  ,\theta) &\mapsto \sum_{j = 1}^q (-1)^j \theta(dx_{i_j})  dx_{i_1} \wedge \ldots \hat{dx_{i_j}} \ldots \wedge dx_{i_q} 
\end{align*}
and 
\begin{align*}
	\tilde{\gamma}_F\colon S^{q}(F) \times F^* &\to S^{q-1}(F) \\
	(y^{\lambda}, \mu) &\mapsto \sum_{i=1}^{r} y^{\lambda - e_i}  \mu(y_i).
\end{align*}
These contraction maps induce morphisms 
$$
\id_F \otimes \tilde{\gamma}_G\colon K_p^q \otimes K_{-1}^0 \to K_{p-1}^{q}
$$
and 
$$
\id_G \otimes \tilde{\gamma}_F\colon K_p^q \otimes K_{-1}^{1} \to K_{p-1}^{q+1}.
$$
We define 
$$
\tilde{\gamma}^q\colon \Tot^q\left( K_p^\bullet \otimes K_{-1}^\bullet \right) = K_p^{q-1} \otimes K_{-1}^1 \oplus  K_p^q \otimes K_{-1}^0  \to K_{p-1}^{q} 
$$
as $$\tilde{\gamma}^q = \id_G \otimes \tilde{\gamma}_F \oplus (-1)^{q} \id_F \otimes \tilde{\gamma}_G.$$

These maps are compatible with the differentials and induce $\gamma$; for details on the proof see  \cite{LichtPdh2022}.

\begin{Lemma} \label{lem:pairing_of_complexes} 
	The maps above define a pairing of complexes 
	$$
	\tilde{\gamma}\colon \Tot^\bullet \left( K_p^\bullet \otimes K_{-1}^\bullet \right) \to K_{p-1}^\bullet
	$$
	and induce the contraction pairing, i.e., given sections $\theta$ of $ \Theta_U^1 $ and $\omega' $ of $ \bigwedge^p G $ with $ \omega := \left(\bigwedge^p\pi\right)(\omega')$, we have 
	$$ \gamma(\omega,\theta) = \left(\bigwedge^{p-1}\pi\right) \circ \tilde{\gamma}(\omega',\pi^*(\theta)).$$ 
\end{Lemma}

\subsection*{The pairing of the total \v Cech complexes}

Let $ \mathcal{U}$ be an open affine covering of $U$. For $p\in \{-1,1,\dots, n-c \}$, let $ \check{C}^\bullet(\mathcal{U}, K_p^\bullet) $ be the \v Cech double complex (as defined in \cite[Tag~01FP]{stkpr}) and let $$ L_p^\bullet = \Tot^\bullet \left(\check{C}^\bullet(\mathcal{U} , K_p^\bullet)\right) $$ be the associated total complex. We consider the cup product map of complexes 
$$
\cup\colon \Tot^\bullet(L_p^\bullet \otimes L_{-1}^\bullet) \to 
\Tot^\bullet(\check{C}^\bullet(\mathcal{U}, \Tot^\bullet(K_p^\bullet \otimes K_{-1}^\bullet) ))
$$
as defined in \cite[Tag~07MB]{stkpr} and compose it with the map 
$$
\Tot^\bullet(\check{C}^\bullet(\mathcal{U}, \Tot^\bullet(K_p^\bullet \otimes K_{-1}^\bullet) )) \to \Tot^\bullet(\check{C}^\bullet(\mathcal{U}, K_{p-1}^\bullet))
$$
induced by the pairing of complexes from Lemma \ref{lem:pairing_of_complexes} to obtain a pairing of complexes
$$
\bar{\gamma}\colon \Tot^\bullet(L_p^\bullet \otimes L_{-1}^\bullet)  \to L_{p-1}^\bullet.
$$


\section{Cohomology for weighted complete intersections}

In this section, we explain how to calculate the cohomology of certain coherent sheaves on weighted complete intersections. We give an overview of results on that matter found in \cite{Dol82} and \cite[Section~8]{Fle81}. We start with weighted projective space, where a similar statement can be found in \cite{Dol82}. The proof for the case of usual projective space, found in \cite[Theorem~III.5.1]{Hartshorne2008}, also works in the general case.

\begin{Lemma}\label{L:cohomology_on_U_for_wps}
	Let $k$ be a field, let $W\in \NN^{n+1}$ be weights, let $ {S_W} = k[x_0,\dots,x_n] $ be the weighted polynomial algebra and let $\PP = \PP_k(W) = \Proj {S_W}$ be weighted projective space.
	Then the following statements hold.
	\begin{enumerate}
		\item The natural map $ {S_W} \to \bigoplus_{l\in\ZZ} H^0(\PP,\Oc_\PP(l)) $ is an isomorphism of graded $S_W$-modules.
		\item We have $ \Hrm^q(\PP,\Oc_\PP(l)) = 0 $ for $ q \neq 0,n $ and $l \in \ZZ$.
		\item In \v{C}ech cohomology with respect to the covering $ \mathcal{U} = \{ D_+(x_i) \} $, we have 
		$$
		\bigoplus_{l\in\ZZ}
		\check{\mathrm{H}}^n(\mathcal{U},\PP,\Oc_\PP(l)) = \coker\left( k\langle x_0^{\alpha_0}\dots x_n^{\alpha_n} \, \middle \vert \, \textrm{ there exists } i \textrm{ with } \alpha_i \geq 0 \rangle \to {S_W}[1/x_0,\dots,1/x_n]  \right).
		$$
	\end{enumerate} 
\end{Lemma}

To better handle the top cohomology, we introduce the $k$-dual module.

\begin{Def}
	Let $k$ be a field, let $A$ be a $k$-algebra and let $M$ be a graded $A$-module.
	We define the \emph{$k$-dual module of M} to be the graded $A$-module $ \Drm(M) = \bigoplus_{l\in \ZZ} \Drm(M)_l $
	with $ \Drm(M)_l = \Hom_k(M_{-l},k)$.
\end{Def}

For example, if $ A = {S_W} $ is a weighted polynomial algebra, then 
$ \Drm({S_W})_l = \Hom_k(({S_W})_{-l},k) $. Here $ ({S_W})_{-l} $ is spanned by the monomials $x_0^{\alpha_0}\dots x_n^{\alpha_n} $ with $ \sum \alpha_i W_i = - l$.
We denote the corresponding dual basis elements by $ \phi_{\alpha_0,\dots,\alpha_n} \in \Drm({{S_W}})_l$. 

Note that $\Drm $ defines a contravariant additive self-functor. 
If we assume $A$ to be finitely generated over $k$ (hence noetherian) and restrict $ \Drm $ to the category of finitely generated graded $A$-modules, then it is exact. This is because in that case, the homogeneous components $M_l$ are finite-dimensional $k$-vector spaces. In particular, under the application of $ \Drm $, injections become surjections, kernels become cokernels and vice versa.    

\begin{Rem}\label{r:cohomology_on_U_for_wps}
	Let $ \vWv = \sum W_i$.
	The $k$-vector space $ \check{\mathrm{H}}^n(\mathcal{U},\PP(W),\Oc_{\PP(W)}(l)) $ vanishes if $l > - \vWv $.
	If $ l \leq - \vWv $, then the vector space is spanned by the elements $x_0^{-1-\alpha_0}\dots x_n^{-1-\alpha_n} $ where $\alpha \in (\ZZ_{\geq 0})^{n+1}$ and $ - \vWv - \sum \alpha_i w_i = l$.
	The $k$ linear map 
	\begin{align*}
		\bigoplus_{l\in\ZZ}
		\check{\mathrm{H}}^n(\mathcal{U},\PP(W),\Oc_{\PP(W)}(l)) \to \Drm(S) (\vert W \vert)
	\end{align*}
	that maps $ x_0^{-1-\alpha_0}\dots x_n^{-1-\alpha_n} $ to $ \phi_{\alpha_0,\dots,\alpha_n} $
	defines an isomorphism of graded $S$-modules. 
\end{Rem}

Consider a complete intersection $X = V(f_1,\dots,f_c) \sbe \PP(W) $ of codimension $c$ and degree $d_1,\dots,d_c$ in $\PP(W) $. 
The surjection of coordinate rings $ {{S_W}} \to A = {{S_W}}/(f_1,\dots,f_c)  $ naturally induces an embedding
$ \Drm(A) \sbe \Drm(S)$. For $ r \in \{1,\dots,c\} $ the scheme
$$
X_r = \Proj A_r, \quad A_r = {{S_W}}/(f_1,\dots,f_r)
$$
is a weighted complete intersection of codimension $r$. We have a chain of closed immersions
$$
X = X_c \sbe \dots \sbe X_0 := \PP(W).
$$
For every $ l \in \ZZ$ and $1 \leq r \leq c-1$, the ideal sheaf sequence in $S_W$
$$
0 \to \Oc_{X_r}(l-d_{r+1}) \xrightarrow{\cdot f_{r+1}} \Oc_{X_r}(l) \to \Oc_{X_{r+1}}(l) \to 0 
$$
is exact, as $f_1,\dots,f_{r+1} $ is a regular sequence.
Considering the associated long exact cohomology sequence and arguing inductively, we can prove the following lemma. (The induction starts with Lemma \ref{L:cohomology_on_U_for_wps} and Remark \ref{r:cohomology_on_U_for_wps}.)
For more details on the proof, we refer to \cite{LichtPdh2022}.

\begin{Lemma} \label{L:cohomology_wpci} Let $ l \in \ZZ $ and let $ X $ be a weighted complete intersection of codimension $c$ as above. Let $ \nu =  \vert W \vert - \sum_{i=1}^c d_i  $.
	Suppose $\dim(X) = n-c \geq 1 $. Then
	\begin{enumerate}
		\item the natural map $ A \to \bigoplus_{l\in\ZZ} \Hrm^0(X,\Oc_{X}(l)) $ is an isomorphism of graded $A$-modules,
		\item $ \mathrm{H}^q(X,\Oc_{X}(l)) = 0 $ for $ q \neq 0,n-c $ and $ l \in \ZZ$, and
		\item $ \bigoplus_{l\in\ZZ} \mathrm{H}^n(X,\Oc_{X}(l)) \cong \Drm(A)  \left( \nu \right) $.
	\end{enumerate} 
\end{Lemma}

\begin{Rem}\label{r:cohomology_wpci}
	If $ A $ is the coordinate ring of a weighted complete intersection $X$ with affine cone $ U = Y \setminus \{0\}$, where $Y= \Spec A$, and $M$ is a graded $A$-module, then there is a natural isomorphism of graded $A$-modules
	$$
	\Hrm^q(U,M^{\sim}\vert_U) \cong \bigoplus\limits_{l\in\ZZ} \Hrm^q(X,(M(l))^\sim),
	$$ 
	where $ M(l) $ denotes the module $ M $ with grading shifted by $ l $, and $ (\_)^\sim$ denotes the functor that associates to an $A$-module its associated $ \Oc_Y $-module (respectively its associated graded $ \Oc_X$-module). This isomorphism can be established by compairing the \v{C}ech cohomology with respect to the coverings $ \{ D(x_i) \} $ for $U$ and $ \{ D_+(x_i) \} $ for $X$ (see \cite{Licht2022}) or with methods of local cohomology (see \cite[Section~8]{Fle81}). We can use this identification to bring the results above in a more compact form. In particular, we have
	$$\Hrm^q(U,\Oc_U) \cong \bigoplus_{l\in\ZZ} \mathrm{H}^q(X,\Oc_{X}(l)). $$
\end{Rem}

\section{The Jacobi ring of a weighted complete intersection}
\label{s:Jacobi ring}

In this section, we will introduce the Jacobi ring of a weighted complete intersection and explain how cohomology can be expressed in terms of it. Our methods build on Flenner's calculation in \cite[Section 8]{Fle81}.
We continue with notations and conventions from Section \ref{s:pairing of resolutions}.
As all components of the complexes $ K_p^\bullet $ are free, the homology of the associated total complex of the \v{C}ech double complexes with respect to the affine covering $\mathcal{U}$ calculates the hypercohomology of these complexes, i.e.,
\begin{align*}
	\mathbb{H}^q(U,K_p^\bullet) = \mathrm{H}^q(\Tot^\bullet (\check{C}^\bullet(\mathcal{U} , K_p^\bullet))) = \mathrm{H}^q(L_p^{\bullet});
\end{align*}
see \cite[Tag~0FLH]{stkpr}.
The total complex associated to a double complex comes with two filtrations $F_1$ and $F_2$ given by
$$
F_1^r( L_p^q) = \bigoplus_{i+j=q, i\geq r} \check{C}^i(\mathcal{U},K_p^j)
$$
and
$$
F_2^r( L_p^q) = \bigoplus_{i+j=q, j\geq r} \check{C}^i(\mathcal{U},K_p^j);
$$
see \cite[Tag~012X]{stkpr}.
The pairing $\bar{\gamma}$ is compatible with these filtrations.
Hence, by Theorem \ref{t:spectral pairing}, we get pairings of the associated spectral sequences, one for each filtration. 
We denote the spectral sequences associated to the filtered complex $ (L_p^\bullet,F_i) $  by 
$ (E^{\bullet,\bullet}_{i,p,r})_{r\in \ZZ_{\geq 0}}$.
See Section \ref{s:spectral pairing} or \cite[Tag~0130]{stkpr} for formulas for the computation of the pages of these spectral sequences.
We first compute the pairing of spectral sequences associated to the filtration $F_1$. 
By Theorem \ref{t:Kp is exact}, on the first page, we see 
$$
E_{1,p,1}^{s,t} = \mathrm{H}^t(\check{C}^s(\mathcal{U},K_p^\bullet)) = \begin{cases}
	\check{C}^s (\mathcal{U},\Omega_U^p) & \text{if}\ t=0 \\
	0 & \text{otherwise}
\end{cases}
$$
if $p>0$ and
$$
E_{1,-1,1}^{s,t} = \mathrm{H}^t(\check{C}^s(\mathcal{U},K_{-1}^\bullet) = \begin{cases}
	\check{C}^s (\mathcal{U},\Theta_U^1) & \text{if}\ t=0 \\
	0 & \text{otherwise.}
\end{cases}
$$
Therefore, all spectral sequences converge on the second page with 
$$
E_{1,p,\infty}^{s,t} = E_{1,p,2}^{s,t} = \begin{cases}
	\mathrm{H}^s (U,\Omega_U^p)) & \text{if}\ t=0 \\
	0 & \text{otherwise}
\end{cases}
$$
and
$$
E_{1,-1,\infty}^{s,t} = E_{1,-1,2}^{s,t}  =  \begin{cases}
	\mathrm{H}^s (U,\Theta_U^1) & \text{if}\ t=0 \\
	0 & \text{otherwise}
\end{cases}
$$
By Theorem \ref{t:spectral pairing}, there is a pairing induced by $\bar{\gamma}$
$$
\Tot^\bullet(E_{1,p,\infty}^{\bullet,\bullet} \otimes E_{1,-1,\infty}^{\bullet,\bullet} ) \to E_{1,p-1,\infty}^{\bullet,\bullet}.
$$
In particular, we obtain a pairing
$$
\mathrm{H}^{s_1}(U,\Omega_U^p) \otimes \mathrm{H}^{s_2}(U,\Theta_U^1) \to \mathrm{H}^{s_1+s_2}(U,\Omega_U^{p-1}).
$$
We note that it is the pairing induced by the contraction map $ \gamma\colon \Omega_U^p \otimes \Theta_U^1 \to \Omega_U^{p-1} $ on cohomology, see Lemma \ref{lem:pairing_of_complexes}. 
Now, we compute the pairing of spectral sequences associated to the filtration $F_2$.
On the first page, we see 
$$
E_{2,p,1}^{s,t} = \mathrm{H}^t(\check{C}^\bullet(\mathcal{U},K_p^s)) = \mathrm{H}^t(U,K_p^s).
$$
All modules involved in the complex $K_p^\bullet$ are free. So by Lemma \ref{L:cohomology_wpci}, we see that the spectral sequence satisfies $ E_{2,p,1}^{s,t} =  0$ if $t \neq 0,n-c$.
We made the assumption that $ p < n-c $. Hence, we see that the spectral sequences converge on page 2 since the differential never connects non-vanishing parts on later pages. We have
$$
E_{2,p,\infty}^{s,t} = E_{2,p,2}^{s,t} = \begin{cases}
	\rH^s(\rH^t(U,K_p^\bullet)) & \text{if } t \in \{0,n-c\} \\
	0 & \text{otherwise}
\end{cases}.
$$
The pairing 
$$
\Tot^\bullet(E_{2,p,\infty}^{\bullet,\bullet} \otimes E_{2,-1,\infty}^{\bullet,\bullet} ) \to E_{2,p-1,\infty}^{\bullet,\bullet}
$$
is the one induced by the pairing 
$$
\Tot^\bullet(K_p^\bullet \otimes K_{-1}^\bullet) \to K_{p-1}^\bullet
$$
on cohomology. Note that for both filtrations, all spectral sequences converge in such a way that for each integer $m$ there is only one combination of $(s,t)$ depending on $m$ such that  $s+t = m$ and 
$$ \gr^s \rH^m(L_p^\bullet) = E_{i,p,\infty}^{s,t} \neq 0; $$
see Equation (\ref{e:E infinity is graded})). 
That means $$ F^q \rH^m(L_p^\bullet) = \begin{cases} 0 &\text{if } q > s \\
	\rH^m(L_p^\bullet)  &\text{if } q \leq s
\end{cases}$$ and therefore in the diagram 
$$
\rH^m(L_p^\bullet) \xleftarrow{\alpha^{m,s}}  F^s \rH^m(L_p^\bullet) \xrightarrow{\beta^{m,s}} \gr^s \rH^m(L_p^\bullet),
$$
the maps $\alpha^{m,s}$ and $\beta^{m,s}$  are both isomorphisms. We combine Diagram (\ref{e:diagram}) for suitable choices of $i$ and $j$ with the diagram from Theorem \ref{t:spectral pairing} to get a commutative diagram 
$$ 
\begin{tikzcd}
	E_{1,p,\infty}^{n-c-p,0} \ar[d] \ar[r] \otimes E_{1,-1,\infty}^{1,0} & \rH^{n-c-p}(L_p^\bullet) \otimes \rH^1(L_{-1}^\bullet) \ar[d]&  \ar[l] E_{2,p,\infty}^{-p,n-c} \otimes E_{2,-1,\infty}^{1,0} \ar[d]\\
	E_{1,p-1,\infty}^{n-c-p+1,0} \ar[r] & \rH^{n-c-p+1}(L_{p-1}^\bullet)	& \ar[l] E_{2,p-1,\infty}^{1-p,n-c}
\end{tikzcd} 
$$ 
where the horizontal morphisms are isomorphisms. Thus, we have identified the pairings of spectral sequences for the filtrations $F_1$ and $F_2$ with each other.
As shown above, the pairing
$$
E_{1,p,\infty}^{n-c-p,0} \otimes E_{1,-1,\infty}^{1,0} \to E_{1,p-1,\infty}^{n-c-p+1,0}
$$
is identified with the contraction map 
$$
\mathrm{H}^{n-c-p}(U,\Omega_U^p) \otimes \mathrm{H}^{1}(U,\Theta_U^1) \to \mathrm{H}^{n-c-p+1}(U,\Omega_U^{p-1}).
$$
On the other hand, the pairing 
$$
E_{2,p,\infty}^{-p,n-c} \otimes E_{2,-1,\infty}^{1,0} \to E_{1,p-1,\infty}^{1-p,n-c}
$$
is the pairing
$$
\rH^{-p}(\rH^{n-c}(U,K_p^\bullet)) \otimes \rH^{1}(\rH^{0}(U,K_{-1}^\bullet)) \to \rH^{1-p}(\rH^{n-c}(U,K_{p-1}^\bullet))
$$
induced by $\tilde{\gamma}$.
We now explicitly calculate all the cohomology groups involved in this pairing. The group $\rH^{-p}(\rH^{n-c}(U,K_p^\bullet)$ is the kernel of the map 
$$
\Hrm^{n-c}(U,K_p^{-p}) \to \Hrm^{n-c}(U,K_p^{1-p}). 
$$
We compute:
$$
\Hrm^{n-c}(U,K_p^{-p}) = \Hrm^{n-c}(U,(S^p(F))) =\bigoplus\limits_{\sum \beta_i = p} \Hrm^{n-c}(U,\OU) \cdot y^{\beta},
$$
$$
\Hrm^{n-c}(U,K_p^{1-p}) = \bigoplus\limits_{i=0}^n \bigoplus\limits_{\sum \beta_i = p-1} \Hrm^{n-c}(U,\OU) \cdot y^{\beta} dx_i. 
$$
We note that $ \deg(y^{\beta}) = \sum \beta_i d_i $, and $\deg(dx_i) = W_i$. Hence, by Lemma \ref{L:cohomology_wpci} and Remark \ref{r:cohomology_wpci}, we see that
$$
\Hrm^{n-c}(U,K_p^{-p}) = \bigoplus\limits_{\sum \beta_i = p} \Drm \left(A\left(-\nu+\sum \beta_i d_i \right)\right) \cdot y^{\beta},
$$ 
and that 
$$
\Hrm^{n-c}(U,K_p^{-p}) = \bigoplus\limits_{i=0}^n \bigoplus\limits_{\sum \beta_i = p-1} \Drm\left(A\left(-\nu+W_i+\sum \beta_i d_i\right)\right) \cdot y^{\beta} dx_i.
$$
%
By Lemma \ref{l:D is exact},
the kernel of the map $ \Hrm^{n-c}(U,K_p^{-p}) \to \Hrm^{n-c}(U,K_p^{1-p}) $ is the k-dual of the cokernel of the map 
$$
\alpha\colon \bigoplus\limits_{i=0}^n \bigoplus\limits_{\sum \beta_i = p-1} A\left(-\nu+W_i+\sum \beta_i d_i\right) \cdot y^{\beta} dx_i \to \bigoplus\limits_{\sum \beta_i = p}  A\left(-\nu+\sum \beta_i d_i \right) \cdot y^{\beta}
$$
that maps $ y^{\beta}dx_i $ to $ \sum_j \partial_{x_i}(f_j) y^{\beta+e_j} $. To describe the cokernel of this map, we introduce the Jacobi ring. 

\begin{Def}
	Let $X = V(f_1, \dots, f_c) \sbe \PP(W) $ be a weighted complete intersection of multidegree $(d_1,\dots,d_c)$. Let $ k[x_0,\dots,x_n,y_1,\dots,y_n] $ be the polynomial ring with bigrading $ \deg(x_i) = (0,w_i)$, $\deg(y_j) = (1,-d_j)$.
	The polynomial $F = y_1f_1+ \dots+ y_cf_c$ is bihomogeneous of degree $(1,0)$.
	We define the \emph{Jacobi ring of Y} to be the bigraded ring 
	$$
	R = k[x_0,\dots,x_n,y_1,\dots,y_c]/(\partial_{x_0}(F),\dots,\partial_{x_n}(F),\partial_{y_0}(F),\dots,\partial_{y_c}(F)).
	$$
\end{Def}

We see that 
$ \coker(\alpha) $ is the part of $R$ in which we fix the first degree to be $p$. In fact, if we view this part $R_{p,*} $ as a graded module via $\deg_2$, we get an isomorphism 
$$
\coker(\alpha) \cong R_{p,*}(-\nu)
$$ 
of graded modules. This shows 
\begin{align*}
	\rH^{-p}(\rH^{n-c}(U,K_p^\bullet)) &= \ker(\Hrm^{n-c}(U,K_p^{-p}) \to \Hrm^{n-c}(U,K_p^{1-p}))\\
	&= \Drm(\coker(\alpha)) \\ &= \Drm(R_{p,*})(\nu). 
\end{align*}

Next we calculate $ \rH^{1}(\rH^{0}(U,K_{-1}^\bullet)) $. It is the cokernel of the map 
$$
\Hrm^0(G^*) = \bigoplus_{i=0}^n A(W_i) \cdot \delta_{i} \xrightarrow{\phi^*}   \Hrm^0(F^*) = \bigoplus_{j=1}^c A(d_j) \cdot y_j^*,
$$
where the differential maps $ \delta_i $ to $ \sum_{j=1}^c \partial_{x_i}(f_j) \cdot y_j^*$. By Lemma \ref{L:cohomology_wpci} and Remark \ref{r:cohomology_wpci}, we can identify $F^*$ with $ \Hrm^0(U,F^*)$ and $G^*$ with $\Hrm^0(U,G^*)$. Hence it can be identified with the $\deg_1 = 1$ part of the Jacobi ring, namely 
$$
\rH^{1}(\rH^{0}(U,K_{-1}^\bullet))  = R_{1,*}.
$$ 
For $x \in R_{1,*}$, let $m_x\colon R_{p-1,*} \to R_{p,*} $ be the multiplication by $x$.
Now under these identifications, the pairing is explicitly given as 
$$
\Drm(R_{p,*})(\nu) \otimes R_{1,*} \to \Drm(R_{p-1,*})(\nu)\colon \quad \varphi \otimes x \mapsto \varphi \circ m_x.
$$
We have proven the following. 

\begin{Lemma}\label{l:contraction_is_multiplication}
	Let $A$ be the coordinate ring of a quasi-smooth weighted projective complete intersection $X = V(f_1, \dots, f_c) \sbe \PP(W_0,\dots,W_n) $ of degree $(d_1,\dots,d_c)$ with affine cone $U$. Let $ \nu = \sum W_i - \sum d_j $. There are isomorphisms $ \mathrm{H}^{n-c-p}(U,\Omega_U^p) \cong \Drm(R_{p,*})(\nu) $ and
	$ \mathrm{H}^{1}(U,\Theta_U^1) \cong R_{1,*}	$. Under theses isomorphisms the contraction pairing
	$$
	\mathrm{H}^{n-c-p}(U,\Omega_U^p) \otimes \mathrm{H}^{1}(U,\Theta_U^1) \to \mathrm{H}^{n-c-p+1}(U,\Omega_U^{p-1})
	$$
	is the pairing
	$$
	\Drm(R_{p,*})(\nu) \otimes R_{1,*} \to \Drm(R_{p-1,*})(\nu)\colon \quad \varphi \otimes x \mapsto \varphi \circ m_x
	$$
\end{Lemma}

\begin{Rem}\label{r:contraction_is_multiplication}
	Giving the pairing $$
	\mathrm{H}^{n-c-p}(U,\Omega_U^p) \otimes \mathrm{H}^{1}(U,\Theta_U^1) \to \mathrm{H}^{n-c-p+1}(U,\Omega_U^{p-1}).
	$$is equivalent to giving a map
	$$
	\mathrm{H}^{1}(U,\Theta_U^1) \to \Hom(\mathrm{H}^{n-c-p}(U,\Omega_U^p), \mathrm{H}^{n-c-p+1}(U,\Omega_U^{p-1}
	)).
	$$
	Under the identifications given in Lemma \ref{l:contraction_is_multiplication}, this is the map 
	$$
	R_{1,*} \to \Hom(\Drm(R_{p,*})(\nu),\Drm(R_{p-1,*})(\nu)) = \Hom(R_{p-1,*}(-\nu),R_{p,*}(-\nu))
	$$
	that sends a homogeneous element $x\in R_{1,*}$ to $m_x$. 
\end{Rem}

\section{Hodge structure on V-varieties} \label{s:v varieties} 

Following \cite[Section 2.5]{PeS08}, we recall some facts about almost Kähler V-varieties (e.g., quasi-smooth weighted complete intersections).

\begin{Def}
	A complex analytic space $X$ is an \emph{$n$-dimensional V-manifold} if there is an open cover $ X = \bigcup X_i $ such that $ X_i = U_i/G_i $ is the quotient of an open subset $ U_i \sbe \CC^n $ by a finite group $G_i $ acting holomorphically on $X_i$. 
	A V-manifold $X$ is \emph{almost Kähler} if there exists a manifold $Y$ that is bimeromorphic to a Kähler manifold and a \emph{proper modification} $f:Y \to X$, i.e., a proper holomorphic map which is biholomorphic over the complement of a nowhere dense analytic subset.  
\end{Def}

There are generalized sheaves of differentials on almost K\"ahler $V$-manifolds.

\begin{Def}
	Let $ X $ be a $V$-manifold. Let $ i\colon X_{sm} \to X $ be the inclusion map of the smooth locus. Define 
	$$ 
	\tilde{\Omega}^p_X = i_* \Omega^p_{X_{sm}}.
	$$ 
\end{Def}

The cohomology of these sheaves determines a Hodge structure, which coincides with the usual Hodge decomposition in the compact K\"ahler case; see \cite[Theorem~2.43]{PeS08} and its proof.

\begin{Th}\label{t:hodge}
	Let $ X $ be an almost K\"ahler  $ V$-manifold. Then, the complex $ \tilde{\Omega}_X^\bullet $ is a resolution of the constant sheaf $ \CC_X $.
	Furthermore the spectral sequence in hypercohomology 
	$$
	E_1^{p,q} = \Hrm^q(X,\tilde{\Omega}^q_X) \Rightarrow \Hrm^{p+q}(X,\CC)
	$$ 
	degenerates on page 1, and $  \Hrm^l({X,\QQ}) $ admits a Hodge structure of weight $l$ given by
	$$ \Hrm^l(X,\QQ) \otimes \CC = \Hrm^l(X,\CC) = \bigoplus_{p+q=l} \Hrm^q(X,\tilde{\Omega}^q_X). $$ 
\end{Th}

As remarked in \cite[Section 7]{Fle81} there are multiple equivalent ways of defining the $\tilde{\Omega}^p_X $. For us, the identification with the reflexive hull of the usual sheaf of differentials is relevant. 

\begin{Lemma}\label{l:omega tilde is omega bidual}
	Let $k$ be an algebraically closed field and let $ X $ be a normal integral scheme of finite type over $k$ and let $i\colon X_{sm}  \to X$ denote the inclusion of the smooth locus. Then there is a canonical isomorphism 
	$$
	(\Omega_X^p)^{**} \to  i_* \Omega_{X_{sm}}^p .
	$$ 
\end{Lemma}

\begin{proof}The restriction map $ \Omega_X^p \to i_* \Omega_{X_{sm}}^p $ induces a map of the corresponding reflexive hulls  $ (\Omega_X^p)^{**} \to (i_*\Omega_{X_{sm}}^p)^{**} $.
	As $ \Omega_{X_{sm}}^p $ is reflexive, there is a canonical isomorphism $(i_*\Omega_{X_{sm}}^p)^{**} \cong i_*\Omega_{X_{sm}}^p  $. The induced map 
	$$ (\Omega_X^p)^{**} \to i_*\Omega_{X_{sm}}^p  $$ is a map of reflexive sheaves that restricted to $X_{sm}$ is an isomorphism. Note since $X$ is normal, the complement of the smooth locus $X\setminus X_{sm} $ has a codimension of at least $2$. Hence it is an isomorphism by \cite[Proposition~1.6]{HartshorneReflexive1980}.
\end{proof}

\begin{Rem}\label{r:theta tilde}
	If $ X \sbe \PP_\CC(W) $ is a quasi-smooth weighted projective variety, then $X$ is normal; see \cite[Proposition~1.3.3]{Dol82} for the case $ X = \PP_\CC(W) $, the argument given there, namely that $X$ is the quotient of its smooth (and hence normal) affine cone by a finite group, also applies in the general case. Hence by Lemma \ref{l:omega tilde is omega bidual} the generalized sheaf of differentials $ \tilde{\Omega}^p_X $ is canonically isomorphic to the reflexive hull $ (\Omega_X^p)^{**} $.
	In particular, the tangent sheaf $ \Theta^1_X $ is therefore canonically isomorphic to the dual $\tilde{\Theta}^1_X  := (\tilde{\Omega}^1_X)^*$ of $ \tilde{\Omega}^1_X $. 
\end{Rem}


\section{Infinitesimal Torelli for weighted complete intersections}\label{s:inf Torelli for wci}

In this section, we proof Theorem \ref{t:inf_torelli_map_for_wci}. We continue with notations from Section \ref{s:Jacobi ring}. From now on we choose the base field $k=\CC$.

Let $X = V(f_1, \dots, f_c) \sbe \PP_\CC(W_0,\dots,W_n) $ be a weighted complete intersection of  degree $(d_1,\dots,d_c)$ with affine cone $U$.
Let $A= S_W/(f_1,\dots,f_c)$ be its coordinate ring. Let $Y = \Spec A$, and let $U= Y \setminus \{0\} $ be the affine cone. Let $ \Omega^1_A $ be the sheaf of $\CC$-differentials on $A$, and let $\Omega^p_A = \bigwedge^p \Omega_A^1$. 
We define the \emph{Euler map} as the $A$-linear morphism
$$
\xi\colon \Omega_A^{p} \to \Omega_A^{p-1}
$$
that sends $ dx_{i_1}\wedge\ldots\wedge dx_{i_p} $ to $ \sum_{j=1}^p (-1)^j W_j x_j \cdot dx_{i_1}\wedge\ldots \hat{dx_{i_j}} \ldots\wedge dx_{i_p}$. 
The associated $\Oc_Y$-module $ (\Omega_A^p)^\sim $ is the sheaf of $p$-Forms on $Y$. Hence, we see that $$ (\Omega_A^p)^\sim\vert_U  = \Omega^p_Y\vert_U = \Omega_U^p. $$
Therefore by Remark \ref{r:cohomology_on_U_for_wps}, there is a natural isomorphism 
\begin{equation} \label{I:Omega}
	\Hrm^q(U,\Omega_U^p) \cong \bigoplus_{l\in\ZZ} \Hrm^q(X,(\Omega_A^p(l)^\sim).
\end{equation}


\begin{Lemma}[{\cite[Lemma~8.9]{Fle81}}]\label{l:euler}
	For all integers $l\in \ZZ$, the complex $ ((\Omega_A^\bullet(l))^\sim,\xi) $ of sheaves on $X$ is exact and the kernel of $ (\Omega_A^p(0))^\sim \xrightarrow{\xi} (\Omega_A^{p-1}(0))^\sim $ is canonically isomorphic to $ \tilde{\Omega}_X^p$.
\end{Lemma}
Lemma \ref{l:euler} gives us short exact sequences 
\begin{equation}\label{e:short exact sequence}
	0 \to \tilde{\Omega}_X^p \to (\Omega_A^p(0))^\sim \xrightarrow{\xi} \tilde{\Omega}_X^{p-1} \to 0.
\end{equation}

There is the following vanishing result.
\begin{Lemma}[{\cite[Lemma~8.10]{Fle81}}]\label{l:Flenner} In the situation above, the following statements hold.
	\begin{enumerate}
		\item We have $\rH^q(U,((\Omega_A^p(l))^\sim) = 0 $, if $p+q \neq n-c, n-c+1$ and $ 0 < q < n-c $.
		\item The map $\Hrm^0(X,(\Omega_A^p(0))^\sim) \xrightarrow{\xi} \Hrm^0(X, \tilde{\Omega}_X^{p-1})$ is surjective if $p \geq 2 $ and has cokernel isomorphic to $\CC$ if $p=1$.
	\end{enumerate}
\end{Lemma}
These results allow us to calculate the relevant cohomology groups. 
\begin{Lemma}In the situation above, the following identities hold. \label{l:torelli cohomology groups}
	\begin{enumerate}
		\item For $0<p<n-c$:  $$ \Hrm^q(X,\tilde{\Omega}^p_X) = \begin{cases}
			0 &\text{if } 0<q <  n-c-p, q \neq p  \\
			\CC &\text{if } 0<q <  n-c-p, q  = p  \\
			\Hom_\CC(R_{p,-\nu},\CC) &\text{if } q = n-c-p, q \neq p \\
			\CC \oplus \Hom_\CC(R_{p,-\nu},\CC) &\text{if } q = p = n-c-p.
		\end{cases}$$
		\item $$ \Hrm^1(X,{\Theta}_X^{1}) = R_{1,0}. $$
	\end{enumerate}
	
\end{Lemma}
\begin{proof} We first prove (1).
	We argue by induction on $p$. In each step we consider the long exact cohomology sequences associated to the short exact Sequence (\ref{e:short exact sequence}) and use Lemmata \ref{L:cohomology_wpci}, \ref{l:contraction_is_multiplication}, \ref{l:Flenner} and Isomorphism \ref{I:Omega} to compute certain cohomology groups.
	Let $p=1$. We know $ \tilde{\Omega}_X^0 = {A}^\sim $. Hence, it follows Lemma \ref{L:cohomology_wpci} that $\Hrm^q(X,\tilde{\Omega}_X^0) = 0 $ for $0<q<n-c$.
	The long exact sequence is
	$$
	\begin{tikzpicture}[descr/.style={fill=white,inner sep=1.5pt}]
		\matrix (m) [
		matrix of math nodes,
		row sep=1em,
		column sep=2.5em,
		text height=1.5ex, text depth=0.25ex
		]
		{ 0 & \Hrm^0(X,\tilde{\Omega}_X^1 ) & \Hrm^0(X,(\Omega_A^1(0))^\sim) &  \Hrm^0(X,\tilde{\Omega}_X^{0}) \\
			& \Hrm^1(X,\tilde{\Omega}_X^1 )  & 0						& 0 \\
			& \Hrm^{n-c-2}(X,\tilde{\Omega}_X^1 )   & 0		& 0\\
			& \Hrm^{n-c-1}(X,\tilde{\Omega}_X^1 ) 	& \Hom_\CC(R_{1,-\nu},\CC)						& 0.  \\			
		};
		
		\path[overlay,->, font=\scriptsize,>=latex]
		(m-1-1) edge (m-1-2)
		(m-1-2) edge (m-1-3)
		(m-1-3) edge (m-1-4)
		(m-1-4) edge[out=355,in=175]  (m-2-2)
		(m-2-2) edge (m-2-3)
		(m-2-3) edge (m-2-4)
		(m-2-4) edge[out=355,in=175,dashed]  (m-3-2)
		(m-3-2) edge (m-3-3)
		(m-3-3) edge (m-3-4)
		(m-3-4) edge[out=355,in=175]  (m-4-2)
		(m-4-2) edge (m-4-3)
		(m-4-3) edge (m-4-4);
	\end{tikzpicture}
	$$
	The assertion for $p=1$ immediately follows. Now assume that $2\leq p < n-c-p-1 $ and that the result holds for $p-1$. We see the assertion is also true for $p$ by considering the long exact sequence
	$$
	\begin{tikzpicture}[descr/.style={fill=white,inner sep=1.5pt}]
		\matrix (m) [
		matrix of math nodes,
		row sep=1em,
		column sep=2.5em,
		text height=1.5ex, text depth=0.25ex
		]
		{ 0 & \Hrm^0(X,\tilde{\Omega}_X^p ) & \Hrm^0(X,(\Omega_A^p(0))^\sim) &  \Hrm^0(X,\tilde{\Omega}_X^{p-1}) \\
			& \Hrm^1(X,\tilde{\Omega}_X^p )  & 0						& 0 \\
			& \Hrm^{p-1}(X,\tilde{\Omega}_X^p )  & 0						& \CC \\
			& \Hrm^{p}(X,\tilde{\Omega}_X^p )  & 0						& 0 \\
			& \Hrm^{n-c-p-1}(X,\tilde{\Omega}_X^p )   & 0		& 0\\
			& \Hrm^{n-c-p}(X,\tilde{\Omega}_X^p ) 	& \Hom_\CC(R_{p,-\nu},\CC)						& 0.  \\			
		};
		
		\path[overlay,->, font=\scriptsize,>=latex]
		(m-1-1) edge (m-1-2)
		(m-1-2) edge (m-1-3)
		(m-1-3) edge (m-1-4)
		(m-1-4) edge[out=355,in=175]  (m-2-2)
		(m-2-2) edge (m-2-3)
		(m-2-3) edge (m-2-4)
		(m-2-4) edge[out=355,in=175,dashed]  (m-3-2)
		(m-3-2) edge (m-3-3)
		(m-3-3) edge (m-3-4)
		(m-3-4) edge[out=355,in=175]  (m-4-2)
		(m-4-2) edge (m-4-3)
		(m-4-3) edge (m-4-4)
		(m-4-4) edge[out=355,in=175,dashed]  (m-5-2)
		(m-5-2) edge (m-5-3)
		(m-5-3) edge (m-5-4)
		(m-5-4) edge[out=355,in=175]  (m-6-2)
		(m-6-2) edge (m-6-3)
		(m-6-3) edge (m-6-4)
		;
	\end{tikzpicture}
	$$
	Similarly the result is verified in case $p \geq n-c-p-1 $. 
	
	Now we prove (2). If we dualize Sequence \ref{e:short exact sequence} for $p=1$ and consider the associated long exact sequence,
	we get 
	$$
	\Hrm^1(X,\Theta^0_X) \to \Hrm^1(X,\Theta_U^1)_0 \to \Hrm^1(X,\Theta_X^{1}) \to \Hrm^2(X,\Theta^0_X);
	$$
	see Remark \ref{r:theta tilde}.
	Under the assumption that $2 < n-c $, we have $ \Hrm^1(X,\Theta^0_X) = \Hrm^2(X,\Theta^0_X) = 0 $ and therefore 
	$$
	\Hrm^1(X,{\Theta}_X^{1}) = \Hrm^1(X,\Theta_U^1))_0 = R_{1,0}. $$ 
	This proves the lemma.
\end{proof}

\begin{proof}[Proof of Theorem \ref{t:inf_torelli_map_for_wci}]
	The statement is a combination of Lemma \ref{l:torelli cohomology groups}, Lemma \ref{l:contraction_is_multiplication} and Remark \ref{r:contraction_is_multiplication}.
\end{proof}


\section{Infinitesimal Torelli for hyperelliptic Fano threefolds of type (1,1,4)}

\label{s:exlpicit calculation of_inf_torelli_map}

In this section, we will prove Theorem \ref{t:inf torelli for hyperelliptic Fanos} and Theorem \ref{t:atuomorphism action on cohomology of Fanos}.
Any hyperelliptic Fano threefold of Picard rank 1, index 1 and degree 4 over $\CC$ is a weighted complete intersection $$ X = V_+(z^2-f,g) \subset \PP_\CC(1,1,1,1,1,2) = \Proj \CC[x_0,\dots,x_4,z] $$ 
with $ f,g\in \CC[x_0,\dots,x_4] $, $\deg(g)=2$, $\deg(f)=4$; see \cite[Theorem~II.2.2.ii)]{Isk80}.
It is a double cover of the smooth quadric $ V(g) \sbe \PP^4 $ with ramification along the smooth surface $ V(f,g) \sbe \PP^4$.
Since $V(g)$ is a smooth quadric, after a change of coordinates, we may assume $ g = x_0^2+ \dots + x_4^2$.
Write $ h_i = \partial_{x_i}(f)/2 $. Then the Jacobi ring of $X$ is given by 
$$
R =  \CC[x_0,\dots,x_4,z,y_2,y_4]/(f-z^2,g,y_2  x_0 - y_4  h_0, \dots, y_2  x_4 - y_4  h_4, y_4 \cdot z).
$$
We apply Theorem \ref{t:inf_torelli_map_for_wci} to a complete intersection of this type. We calculate $ \nu = 7-6=1 $ and therefore 
$$
\Hrm^1(X,\Theta_X) \cong R_{1,0},\,  \Hrm^{1}(X,\tilde{\Omega}_X^2) \cong \Hom_\CC(R_{1,-1},\CC),\, \Hrm^{2}(X,\tilde{\Omega}_X^1) \cong \Hom_\CC(R_{2,-1},\CC).
$$
There are surjections
$$
y_2 \cdot \CC[x_0,\dots,x_4,z]_2 \oplus y_4  \cdot \CC[x_0,\dots,x_4,z]_4 \to R_{1,0},
$$
$$
y_2 \cdot \CC[x_0,\dots,x_4,z]_1 \oplus y_4  \cdot \CC[x_0,\dots,x_4,z]_3 \to R_{1,-1},
$$
$$
y_2^2 \cdot \CC[x_0,\dots,x_4,z]_3 \oplus y_2y_4  \cdot \CC[x_0,\dots,x_4,z]_5  \oplus y_4^2 \cdot \CC[x_0,\dots,x_4,z]_7  \to R_{2,-1}.
$$
Let $ B = \CC[x_0,\dots,x_4]/(f,g) $. Using the relations $ y_2  x_i = y_4  h_i $ and $ y_4 z = 0 $, we see 
$$
R_{1,0} \cong (y_2 B_2 \oplus \CC \cdot y_2 z \oplus y_4 B_4)/(y_2  x_i - y_4  h_i) \cong \CC \cdot y_2 z \oplus y_4 B_4, 
$$ 
$$
R_{1,-1} \cong \left(y_2 B_1 \oplus y_4 B_3 \right)/(y_2  x_i - y_4  h_i) \cong  y_4 B_3, 
$$ 
$$
R_{2,-1} \cong \left(y_2^2 \cdot B_3 \oplus y_2y_4 B_5 \oplus y_4^2 B_7 \right)/(y_2  x_i - y_4  h_i) \cong y_4^2 B_7.
$$ 
Note that there are injections $$ T_1 := y_2 B_2 \oplus \CC \cdot y_2 z \to R_{1,0} $$ and $$ T_2 := y_2 B_1 \to R_{1,-1}. $$ We will need the following Lemma to prove Theorem \ref{t:atuomorphism action on cohomology of Fanos}.
\begin{Lemma}\label{l:automorphisms}
	If $\varphi \in \Aut(X) $, then there are linear polynomials $ \lambda_i \in k[x_0,\dots,x_4]_1 $ and $b\in \CC^*$ such that, for all 
	$ (x_0:\dots:x_4:z) \in X(\CC) $, we have 
	$$\varphi(x_0:\dots:x_4:z) = (\lambda_0:\dots:\lambda_4:bz).$$
\end{Lemma}
\begin{proof}
	The anticanonical bundle of $X$ is isomorphic to $\Oc_{X}(1)$; see \cite[Theorem~3.3.4]{Dol82}.
	The cohomology group $ \Hrm^0(X,\Oc_{X}(1)) $ is a 5-dimensional vector space generated by $ x_0,\dots,x_4 $, and $\Hrm^0(X,\Oc_{X}(2))  $ is a 15-dimensional vector space generated by $ x_0^2,x_0x_1,\allowbreak \dots,x_4^2,z $. Any automorphism $ \varphi \in  \Aut(X) $ induces an automorphism of these cohomology groups. Hence $\varphi$ is of the form 
	$$
	\varphi(x_0,\dots,x_4,z) = (\lambda_0,\dots,\lambda_4,bz+q),
	$$
	where $ \lambda_i \in k[x_0,\dots,x_4]_1 $, $b\in \CC^*$ and $q\in\CC[x_0,\dots,x_4]_2. $
	Note if $ g(x_0,\dots,x_4)=0$, then there is a $ z\in \CC $ such that $ (x_0,\dots,x_4,z) \in X $. This shows $ g(\lambda_0,\dots,\lambda_4) $ vanishes on $ V_+(g)\sbe \PP^4 $.
	By Hilbert's Nullstellensatz, $ g(\lambda_0,\dots,\lambda_4) = \nu g$ for some $\nu\in \CC^*$. Furthermore, again by Hilbert's Nullstellensatz, we see 
	$$
	(bz+q)^2-f(\lambda_0,\dots,\lambda_4) \in (z^2-f,g).
	$$
	It follows that $g$ divides $q$. As $q$ vanishes on $V_+(g)$, we can put $q=0$.
\end{proof}

In particular, the involution coming from the double cover is given by 
$$
\iota\colon 
X \to X\colon \quad (x_0,\dots,x_4,z) \mapsto (x_0,\dots,x_4,-z). 
$$

\begin{proof}[Proof of Theorem {\ref{t:atuomorphism action on cohomology of Fanos}}]
	(1): Consider an automorphism $ \varphi \in \Aut(X) $ as in Lemma \ref{l:automorphisms}. If $\varphi$ operates trivially on $ \Hrm^1(X,\Theta_X^1) $, then it operates trivially on $ T_1 $. Therefore, we have $b=1$ $\varphi(x_ix_j) = x_ix_j$ for all $i,j$. This shows either $ \lambda_i = x_i $ for all $i$ or $ \lambda_i = -x_i $ for all $i$. Note in $\PP(1,1,1,1,1,2)_\CC $, the coordinates $ (x_0:\ldots:x_4,z) $ and $ (-x_0:\ldots:-x_4,z) $ define the same point. 
	Hence $\varphi = \id$.    
	(2): As mentioned in the introduction, this is already proven; see \cite[Proposition~2.12]{JaLoCI2017}.
	
	(3): If $\varphi$ acts trivially on $ \Hrm^3(X,\CC ) $, then it acts trivially on $\Hrm^{2,1}$.
	In particular, such a $\varphi$ then operates trivially on $T_2$.
	Hence, we have $ \lambda_i = x_i $ for $i\in \{0,\dots,4\}$. As $ \varphi $ has to preserve the equation $ z^2-f $, we see $ b \in \{1,-1\}$. This implies $ \varphi \in \{\id,\iota\}$. 
\end{proof}

\begin{proof}[Proof of Theorem {\ref{t:inf torelli for hyperelliptic Fanos}}]
	From the explicit descriptions above, we calculate that the involution invariant part $ H^1(X,\Theta_X)^\iota$ is
	$$
	(R_{1,0})^\iota \cong y_4 B_2.
	$$ 
	Hence by Theorem \ref{t:inf_torelli_map_for_wci}, the involution invariant infinitesimal Torelli map can be identified with the map 
	$$
	B_4 \to \Hom\left(B_3,B_7\right).
	$$
	The sequence $f,g$ is regular as these polynomials define a complete intersection in $ \PP^4 $. We can find polynomials $h_1,h_2,h_3$ such that $ f,g,h_1,h_2,h_3 $ is regular. 
	Note that we can choose these polynomials of arbitrarily large degrees.
	Now, by Macaulay's theorem \cite[Corollary~6.20]{Voisin2003}, the map 
	$$
	\left(\frac{\CC[x_0,\dots,x_4]}{(f,g,h_1,h_2,h_3)}\right)_4 \to \Hom\left(\left(\frac{\CC[x_0,\dots,x_4]}{(f,g,h_1,h_2,h_3)}\right)_3,\left(\frac{\CC[x_0,\dots,x_4]}{(f,g,h_1,h_2,h_3)}\right)_7\right)
	$$
	is injective.
\end{proof}


\printbibliography[heading=bibintoc]

\end{document}